\documentclass[11pt]{amsart}
\usepackage{amsmath}
\usepackage{cases}
\usepackage{amssymb}
\usepackage{amsfonts}
\usepackage{amssymb,amsfonts}
\allowdisplaybreaks

\begin{document}
\theoremstyle{plain}
\newtheorem{thm}{Theorem}[section]
\newtheorem{theorem}[thm]{Theorem}
\newtheorem{lemma}[thm]{Lemma}
\newtheorem{corollary}[thm]{Corollary}
\newtheorem{corollary*}[thm]{Corollary*}
\newtheorem{proposition}[thm]{Proposition}
\newtheorem{proposition*}[thm]{Proposition*}
\newtheorem{conjecture}[thm]{Conjecture}
\theoremstyle{definition}
\newtheorem{construction}{Construction}
\newtheorem{notations}[thm]{Notations}
\newtheorem{question}[thm]{Question}
\newtheorem{problem}[thm]{Problem}
\newtheorem{remark}[thm]{Remark}
\newtheorem{remarks}[thm]{Remarks}
\newtheorem{definition}[thm]{Definition}
\newtheorem{claim}[thm]{Claim}
\newtheorem{assumption}[thm]{Assumption}
\newtheorem{assumptions}[thm]{Assumptions}
\newtheorem{properties}[thm]{Properties}
\newtheorem{example}[thm]{Example}
\newtheorem{comments}[thm]{Comments}
\newtheorem{blank}[thm]{}
\newtheorem{observation}[thm]{Observation}
\newtheorem{defn-thm}[thm]{Definition-Theorem}

\newcommand{\sM}{{\mathcal M}}


\title[On a proof of the Bouchard-Sulkowski conjecture]{On a proof of the Bouchard-Sulkowski conjecture}
       \author{Shengmao Zhu}
        \address{Department of Mathematics and Center of Mathematical Sciences, Zhejiang University, Hangzhou, Zhejiang 310027, China}
        \email{zhushengmao@gmail.com}

        \begin{abstract}
        In this short note, we give a proof of the free energy part of the BKMP conjecture
        of $\mathbb{C}^3$ proposed by Bouchard and Sulkowski
        \cite{BS}. Hence the proof of the full BKMP conjecture for the case
        of $\mathbb{C}^3$ has been finished.
        \end{abstract}
    \maketitle

\section{Introduction}

Motivated by B. Eynard and his collaborators' series works on matrix
model \cite{Ey1,EO1,EO2}, V. Bourchard, A. Klemm, M. Mari\~no and S.
Pasquetti \cite{BKMP1} proposed a new approach (Remodeling the
B-model) to compute the topological string amplitudes for local
Calabi-Yau manifolds and conjectured that the remodeling approach is
equivalent to the Gromov-Witten theory of corresponding toric
Calabi-Yau manifolds \cite{BKMP2}. In particularly, for the case of
$\mathbb{C}^3$, V. Bouchard and M. Mari\~no \cite{BM} calculated the
correlation functions by remodeling approach and conjectured that
they are equal to the topological vertex computed by Gromov-Witten
theory. Later, L. Chen and J. Zhou \cite{Ch, Zhou} gave the rigorous
proof independently based on the symmetric form of cut-and-join
equation of Mari\~no-Vafa formula proved in \cite{LLZ}(see also
\cite{Ey2} for a new proof).

Recently, V. Bouchard and P. Sulkowski \cite{BS} proposed the
following free energy part of the BKMP conjecture for the case of
$\mathbb{C}^3$ (Conjecture 2 in \cite{BS}).
\begin{conjecture}
Let $\Sigma_f$ be the framed curve mirror to $X=\mathbb{C}^3$. Then
the free energies obtained through the Eynard-Orantin recursion are
given by:
\begin{align*}
F^{(g)}=\frac{1}{2}(-1)^g\times\frac{|B_{2g}||B_{2g-2}|}{2g(2g-2)(2g-2)!}.
\end{align*}
\end{conjecture}

In this note, we will give a proof of the Conjecture 1.1 based on a
Hodge integral identity and some residue calculations. After
finished the draft of this paper and contacted with Prof. V.
Bouchard, the author knew that this conjecture has also been proved
by V. Bouchard and his collaborators \cite{BCMS} at the same time.

\section{The BKMP conjecture}
Let us consider a Riemann surface with genus $\overline{g}$,
\begin{align*}
\Sigma=\{x,y\in \mathbb{C}^*|H(x,y;z_a)=0\}\subset
\mathbb{C}^*\times\mathbb{C}^*
\end{align*}
where $z_a, a=1,..,k$ are the deformation parameters of the complex
structure of $\Sigma$. $H(x,y;z_a)$ is a polynomial in $(x,y)$ which
are $\mathbb{C}^*$-variables.  Let $q_i, i=1,..,2\overline{g}+2$ be
the ramification points of $\Sigma$ and on the neighborhood of
$q_i$, one can find two distinct points $q, \overline{q} \in \Sigma$
such that $x(q)=x(\overline{q})$. We mention that the mirror curve
of a toric Calabi-Yau 3-fold satisfies these conditions.

The recursion process starts with the following ingredients.
\subsection{Ingredients} First, one needs the meromorphic
differential
$
\omega(p)=\log y(p)\frac{dx(p)}{x(p)}
$
on $\Sigma$.

One also needs the Bergmann kernel $B(p,q)$ on $\Sigma$ defined by
the following conditions.
\begin{align*}
&B(p,q)\sim_{p\rightarrow q}\frac{dpdq}{(p-q)^2}+\text{finite}.\\
&B(p,q)\, \text{is holomorphic except $p=q$}.\\
&\oint_{A_\alpha}B(p,q)=0,\ \alpha=1,..,\overline{g}.
\end{align*}
where $(A_\alpha,B^\alpha)$ is a canonical symplectic basis of
one-cycles on $\Sigma$.

For the case of $\overline{g}=0$, the Bergman kernel is given by
\begin{align*}
B(p_1,p_2)=\frac{dy_1dy_2}{(y_1-y_2)^2}, \ y_i=y(x_i).
\end{align*}

\subsection{BKMP's construction}
 Inspired by the work \cite{Mar}, Bouchard, Klemm, Mari\~no and Pasquetti \cite{BKMP1,BKMP2}
 defined the free energy $F^{(g,h)}(p_1,..,p_h)$ on the mirror curve
 $H(x,y;z_a)=0$ based on the topological recursions construced by B. Eynard and N. Orantin \cite{EO1} as follows.
\begin{align*}
&F^{(g,h)}(p_1,..,p_h)=\int W^{(g,h)}(p_1,..,p_h),\\\nonumber
&W^{(0,1)}(p)=\omega(p),\\\nonumber &
W^{(0,2)}(p_1,p_2)=B(p_1,p_2)-\frac{dp_1dp_2}{(p_1-p_2)^2},\\\nonumber
&W^{(g,h)}(p_1,..,p_h)=\tilde{W}^{g,h}(p_1,..,p_h),\ \text{for} \,
(g,h)\neq (0,1),\ (0,2).
\end{align*}
where $\tilde{W}^{(g,h)}(p_1,...,p_h)$ is a multilinear meromorphic
differential defined by the following topological recursions.
\begin{align*}
&\tilde{W}^{(0,1)}(p)=0, \tilde{W}^{(0,2)}(p,q)=B(p,q),\\\nonumber
&\tilde{W}^{(g,h+1)}(p,p_1,..,p_h)=\sum_{q_i}Res_{q=q_i}\frac{dE_{q,\overline{q}}(p)}{\omega(q)-\omega(\overline{q})}
\left(\tilde{W}^{(g-1,h+2)}(q,\overline{q},p_1,..,p_h)\right.\\\nonumber
&\left.\qquad\qquad\qquad\qquad+\sum_{l=0}^g\sum_{J\subset
H}\tilde{W}^{(g-l,|J|+1)}(q,p_J)\tilde{W}^{(l,|H|-|J|+1)}(\overline{q},p_{H\setminus
J})\right), \\\nonumber &H=\{1,..,h\}, J=\{i_1,..,i_j\}\subset H,
p_J=\{p_{i_1},..,p_{i_j}\}, \\\nonumber &
dE_{q,\overline{q}}(p)=\frac{1}{2}\int_{q}^{\overline{q}}B(p,\psi),
\ \text{near a ramification point} \ q_i.
\end{align*}

Moreover, in \cite{BKMP2}, they defined $F^{(g)}$ (\ $g\in
\mathbb{Z},\ g\geq 2$ ) on $\Sigma$ by
\begin{align*}
F^{(g)}=\frac{(-1)^g}{2-2g}\sum_{q_i}
Res_{q=q_i}\theta(q)W^{(g,1)}(q),
\end{align*}
where $\theta(q)$ is any primitive of $\omega(q)$ given by
$d\theta(q)=\omega(q)$. And $F^{(1)}$ is defined separately as
\begin{align*}
F^{(1)}=-\frac{1}{2}\log \tau_B-\frac{1}{24}\log \prod_i
\omega^{'}(q_i),
\end{align*}
where $\omega^{'}(q_i)=\frac{1}{dz_i(p)}d\left(\frac{\log
y(x)}{x}\right)|_{p=q_i}$, \ $z_i(p)=\sqrt{x(p)-x(q_i)}$ and
$\tau_B$ is the Bergmann tau-function \cite{EO1}.

Then the BKMP conjecture for toric Calabi-Yau 3-fold can be
formulated as follow (Conjecture 1 in \cite{BS}).
\begin{conjecture}
Let $\Sigma_f$ be the framed mirror curve to a toric Calabi-Yau
threefold $X$.

1. The free energies $F^{(g)}$ constructed by the Eynard-Orantin
recursion are mapped by the mirror map to the genus g generating
functions of Gromov-Witten invariants of $X$.

2. The correlation functions $F^{(g,h)}$ are mapped by the
open/closed mirror map to the generating functions of framed open
Gromov-Witten invariants.
\end{conjecture}

\section{BKMP conjecture for the case of $\mathbb{C}^3$}
In this section, we restrict us to consider the special toric
Calabi-Yau 3-fold $\mathbb{C}^3$ which has the framed mirror curve
\begin{align*}
\Sigma_{f}=\{H(x,y):=x+y^{f}+y^{f+1}=0\} \subset (\mathbb{C}^{*})^2.
\end{align*}
$\Sigma_f$ has only one ramification point
\begin{align*}
y_*=\frac{-f}{(f+1)},\quad x_*=\frac{f^f}{(-1-f)^{-1-f}}.
\end{align*}

By the definition of $\theta(q)$ given in Section 2,  we have
\begin{align}
\theta(y)=\frac{f}{2}(\log y)^2+\log y \log(1+y)+Li_2(-y).
\end{align}

We define the differential form $\Psi_{n}(y;f)$ for $n\geq 0$ as
follow:
\begin{align}
\Psi_n(y;f)=-dy\frac{((1+f)y+f)}{y(y+1)}\left(\frac{y(y+1)}{(1+f)y+f}\frac{d}{dy}\right)^{n+1}\frac{1}{(1+f)((1+f)y+f)}.
\end{align}
For examples, when $n=0$ and $1$:
\begin{align*}
&\Psi_0(y;f)=dy\frac{1}{(f+(f+1)y)^2};\\
&\Psi_1(y;f)=-dy\frac{3(1+f)y(y+1)-(1+2y)(f+(f+1)y)}{(f+(f+1)y)^4}.
\end{align*}

For convenience, in the following exposition, we also introduce the
notation $\hat{\Psi}_{n}(y;f)$ by the relationship
$\Psi_{n}(y;f)=-\hat{\Psi}_{n}(y;f)dy$.

By the Eynard-Orantin topological recursions introduced in Section
2, we have
\begin{align*}
&W^{(0,3)}(y_1,y_2,y_3)=-(f(f+1))^2\Psi_0(y_1;f)\Psi_0(y_2;f)\Psi_0(y_3;f);\\
&W^{(0,4)}(y_1,y_2,y_3,y_4)=(f(f+1))^3\sum_{i=1}^{4}\Psi_1(y_i;f)\prod_{j\neq
i}\Psi_0(y_j;f);\\
&W^{(1,1)}(y)=\frac{1}{24}\left((1+f+f^2)\Psi_0(y;f)-f(f+1)\Psi_1(y;f)\right);\\
&W^{(1,2)}(y_1,y_2)\frac{1}{24}(-(1+f+f^2)\Psi_0(y_1;f)\Psi_1(y_2;f)\\\nonumber
&+f(1+f)\Psi_0(y_1;f)\Psi_2(y_2;f)+(y_1\leftrightarrow
y_2)+f(1+f)\Psi_1(y_1;f)\Psi_1(y_2;f));\\
&W^{(2,1)}(y)=\frac{1}{5760}\left(2f(f+1)\Psi_1(y;f)-7(1+f+f^2)^2\Psi_2(y;f)\right.\\\nonumber
&\left.+12f(1+2f+2f^2+f^3)\Psi_3(y;f)-5f^2(f+1)^2\Psi_4(y;f)\right).
\end{align*}

For the general $g\geq 2$ and $h\geq 1$,  L. Chen \cite{Ch} and J.
Zhou \cite{Zhou} have proved the following identity independently
(See also \cite{Zhu}).
\begin{align*}
W^{(g,h)}(y_1,..,y_h)=(-1)^{g+h}(f(f+1))^{h-1}\sum_{n_i\geq
0}\langle
\prod_{i=1}^h\tau_{n_i}\Lambda_g^{\vee}(1)\Lambda_g^{\vee}(-f-1)\Lambda_g^{\vee}(f)
\rangle_g\prod_{i=1}^{h}\Psi_{n_i}(y_i;f)
\end{align*}
where
$\Lambda_{g}^{\vee}(t)=t^g-t^{g-1}\lambda_1+\cdots+(-1)^g\lambda_g$.

In particularly,
\begin{align*}
W^{(g,1)}(y)=(-1)^{g+1}\sum_{n\geq 0}^{2g-2}\langle \tau_n
\Lambda_g^{\vee}(1)\Lambda_g^{\vee}(-f-1)\Lambda_g^{\vee}(f)\rangle_g\Psi_b(y;f).
\end{align*}

Thus the free energy part of the BKMP conjecture for the case of
$\mathbb{C}^3$ is given by
\begin{align*}
F^{(g)}&=\frac{(-1)^g}{2-2g}
Res_{y=\frac{-f}{1+f}}\theta(y)W^{(g,1)}(y)\\\nonumber
&=\frac{1}{2g-2}\sum_{n\geq 0}\langle\tau_n
\Lambda_g^{\vee}(1)\Lambda_g^{\vee}(-f-1)\Lambda_g^{\vee}(f)\rangle_g
Res_{y=\frac{-y}{1+y}}\theta(y)\Psi_n(y;f).
\end{align*}

Then Conjecture 1.1 will be finished by the following two lemmas and
the Hodge integral identity \cite{FabP,LLZ2},
\begin{align*}
\langle\lambda_{g}\lambda_{g-1}\lambda_{g-2}\rangle_g=\frac{1}{2(2g-2)!}\frac{|B_{2g-2}|}{2g-2}\frac{|B_{2g}|}{2g}.
\end{align*}

\begin{lemma}
The degree $3g-3$ part of
$\Lambda_g^{\vee}(1)\Lambda_g^{\vee}(-f-1)\Lambda_g^{\vee}(f)$ is
given by
\begin{align*}
(-1)^{g-1}f(f+1)\lambda_g\lambda_{g-1}\lambda_{g-2}.
\end{align*}
\end{lemma}
\begin{proof}
By Mumford's relation:
$\Lambda_{g}^{\vee}(1)\Lambda_{g}^{\vee}(-1)=(-1)^g$, we have
$\lambda_{g-1}^2=2\lambda_{g}\lambda_{g-2}, \lambda_g^2=0$. Then the
degree $3g-3$ part of
$\Lambda_g^{\vee}(1)\Lambda_g^{\vee}(-f-1)\Lambda_g^{\vee}(f)$ is
equal to
\begin{align*}
&((-1)^g\lambda_g+(-1)^{g-1}\lambda_{g-1}+(-1)^{g-2}\lambda_{g-2})\times(-1)^g(\lambda_g+(f+1)\lambda_{g-1}+(f+1)^2\lambda_{g-2})\\\nonumber
&\times((-1)^g\lambda_g+f(-1)^{g-1}\lambda_{g-1}+f^2(-1)^{g-2}\lambda_{g-2})\\\nonumber
&=(-1)^{g-1}f(f+1)\lambda_g\lambda_{g-1}\lambda_{g-2}.
\end{align*}
\end{proof}
\begin{lemma}

\makeatletter
\let\@@@alph\@alph
\def\@alph#1{\ifcase#1\or \or $'$\or $''$\fi}\makeatother
\begin{subnumcases}
{Res_{y=\frac{-f}{1+f}}\theta(y)\Psi_{n}(y;f)=} -\frac{1}{f(1+f)},
&$n=1$, \label{eq:a1}\nonumber \\\nonumber 0, &$n\geq 2$ or
$n=0$.\label{eq:a2}
\end{subnumcases}
\makeatletter\let\@alph\@@@alph\makeatother
\end{lemma}

\begin{proof}
Let $z=y+\frac{f}{1+f}$, by formula $(1)$
\begin{align*}
\theta(z)=\frac{f}{2}\left(\log\left(z-\frac{f}{1+f}\right)\right)^2+\log\left(z-\frac{f}{1+f}\right)\log\left(z+\frac{1}{1+f}\right)
+Li_2\left(-z+\frac{f}{1+f}\right).
\end{align*}
Hence,
\begin{align*}
d\theta(z)=\frac{(1+f)z\log\left(z-\frac{f}{1+f}\right)}{(z-\frac{f}{1+f})(z+\frac{1}{1+f})}dz.
\end{align*}
From formula $(2)$,
\begin{align}
\hat{\Psi}_{n}(z;f)=-\frac{d}{dz}\left(\hat{\Psi}_{n-1}(z;f)\frac{\left(z-\frac{f}{1+f}\right)\left(z+\frac{1}{1+f}\right)}{(1+f)z}\right)
\ \text{for}\ n\geq 1.
\end{align}
and
\begin{align*}
\hat{\Psi}_0(z;f)=-\frac{1}{(1+f)^2z}.
\end{align*}

By the recursion formula $(4)$, it is easy to show that
$\hat{\Psi}_n(z;f)$ has the following form
\begin{align*}
\hat{\Psi}_n(z;f)=\frac{a_0(f)+a_1(f)z+\cdots+a_{2n}(f)z^{2n}}{((1+f)z)^{2n+2}}
\end{align*}
where $a_{0}(f),..,a_{2n}(f)$ are some polynomials of framing $f$.

From the residue identity,
\begin{align*}
0=Res_{z=0}d(f(z)g(z))=Res_{z=0}g(z)df(z)+Res_{z=0}f(z)dg(z)
\end{align*}
We have
\begin{align}
Res_{z=0}g(z)df(z)=-Res_{z=0}f(z)dg(z).
\end{align}

Formula $(5)$ will be used iteratively in the following exposition.

When $n=0$,
\begin{align*}
&Res_{y=\frac{-f}{1+f}}\theta(y)\Psi_{0}(y;f)\\
&=-Res_{y=-\frac{f}{1+f}}\theta(y)\hat{\Psi}_0(y;f)dy\\\nonumber
&=-Res_{z=0}\theta(z)\hat{\Psi}_{0}(z;f)dz\\\nonumber
&=-Res_{z=0}\theta(z)\frac{1}{(1+f)^2}d\left(\frac{1}{z}\right)\\\nonumber
&=Res_{z=0}\frac{1}{(1+f)^2z}d\theta(z)\\\nonumber
&=Res_{z=0}\frac{\log\left(z-\frac{f}{1+f}\right)}{(1+f)\left(z-\frac{f}{1+f}\right)\left(z+\frac{1}{1+f}\right)}\\\nonumber
&=0.
\end{align*}

When $n=1$,
\begin{align*}
&Res_{y=\frac{-f}{1+f}}\theta(y)\Psi_{1}(y;f)\\
&=-Res_{y=-\frac{f}{1+f}}\theta(y)\hat{\Psi}_1(y;f)dy\\\nonumber
&=-Res_{z=0}\theta(z)\hat{\Psi}_{1}(z;f)dz\\\nonumber
&=Res_{z=0}\theta(z)d
\left(\hat{\Psi}_0(z;f)\frac{\left(z-\frac{f}{1+f}\right)\left(z+\frac{1}{1+f}\right)}{(1+f)z}\right)\\\nonumber
&=-Res_{z=0}\hat{\Psi}_0(z;f)\log\left(z-\frac{f}{1+f}\right)dz\\\nonumber
&=Res_{z=0}\frac{1}{(1+f)^2z^2}\log\left(z-\frac{f}{1+f}\right)dz\\\nonumber
&=-\frac{1}{f(1+f)}
\end{align*}

More generally, when $n\geq 2$
\begin{align*}
&Res_{y=\frac{-f}{1+f}}\theta(y)\Psi_{n}(y;f)\\
&=-Res_{y=-\frac{f}{1+f}}\theta(y)\hat{\Psi}_n(y;f)dy\\\nonumber
&=-Res_{z=0}\theta(z)\hat{\Psi}_{n}(z;f)dz\\\nonumber
&=Res_{z=0}\theta(z)d\left(\hat{\Psi}_{n-1}(z;f)\frac{\left(z-\frac{f}{1+f}\right)\left(z+\frac{1}{1+f}\right)}{(1+f)z}\right)\\\nonumber
&=-Res_{z=0}\hat{\Psi}_{n-1}(z;f)\frac{\left(z-\frac{f}{1+f}\right)\left(z+\frac{1}{1+f}\right)}{(1+f)z}
d\theta(z)\\\nonumber &=-Res_{z=0}\hat{\Psi}_{n-1}(z;f)\log\left(z-
\frac{f}{1+f}\right)dz\\\nonumber
&=Res_{z=0}d\left(\hat{\Psi}_{n-2}(z;f)\frac{\left(z-\frac{f}{1+f}\right)\left(z+\frac{1}{1+f}\right)}{(1+f)z}\right)
\log\left(z-\frac{f}{1+f}\right)\\\nonumber
&=-Res_{z=0}\hat{\Psi}_{n-2}(z;f)\frac{\left(z+\frac{1}{1+f}\right)}{(1+f)z}\\\nonumber
&=-Res_{z=0}\frac{a_0(f)+a_1(f)z+\cdots+a_{2n-4}(f)z^{2n-4}}{((1+f)z)^{2n-2}}\left(\frac{1}{(1+f)}+\frac{1}{(1+f)^2z}\right)\\\nonumber
&=0.
\end{align*}
\end{proof}

Now, we can finish the proof of Conjecture 1.1 by lemma 3.1 and 3.2.
\begin{proof}
\begin{align*}
F^{(g)}&=\frac{(-1)^g}{2-2g}Res_{y=-\frac{f}{1+f}}\theta(y)W^{(g,1)}(y)\\\nonumber
&=\frac{1}{2g-2}\sum_{n=1}^{3g-2}\langle\tau_n\Lambda_g^{\vee}(1)\Lambda_{g}^{\vee}(-f-1)\Lambda_g^{\vee}(f)\rangle_g
Res_{y=-\frac{f}{1+f}}\theta(y)\Psi_n(y;f)\\\nonumber
&=-\frac{1}{2g-2}\frac{1}{f(1+f)}\langle\tau_1\Lambda_g^{\vee}(1)\Lambda_{g}^{\vee}(-f-1)\Lambda_g^{\vee}(f)\rangle_g\\\nonumber
&=\frac{(-1)^g}{2g-2}\langle\tau_1\lambda_g\lambda_{g-1}\lambda_{g-2}\rangle_g\\\nonumber
&=(-1)^g\langle\lambda_g\lambda_{g-1}\lambda_{g-2}\rangle_g\\\nonumber
&=(-1)^g\frac{1}{2(2g-2)!}\frac{|B_{2g-2}|}{2g-2}\frac{|B_{2g}|}{2g}
\end{align*}
where we have used the dilaton equation for Hodge integrals
\begin{align*}
\langle\tau_1\lambda_g\lambda_{g-1}\lambda_{g-2}\rangle_g=(2g-2)\langle\lambda_{g}\lambda_{g-1}\lambda_{g-2}\rangle_g.
\end{align*}
Thus the Conjecture 1.1 is proved.
\end{proof}

{\bf Acknowledgements.} The author would like to thank Professor
Kefeng Liu for his interest in this work.

\vskip 30pt
$$ \ \ \ \ $$

\end{document}